\documentclass[12pt]{amsart}

\usepackage{amsmath, amsthm, amssymb}
\usepackage[dvipsnames]{xcolor}
\usepackage{mathtools}
\usepackage{tikz-cd}
\usepackage{pgfplots}
\usepackage{todonotes}
\usepackage{comment}
\usepackage{dsfont}
\usepackage{thm-restate}
\usepackage{pgfplots}
\usepackage{tikz}
\usetikzlibrary{math}
\usetikzlibrary{arrows} 
\usepackage{tikz-3dplot}
\usetikzlibrary{shapes,calc,positioning}
\usepackage{wasysym}
\usepackage{caption}
\usepackage{hyperref}
\usepackage{multirow}
\usepackage[margin=2.5cm]{geometry}
\usepackage{array}
\usepackage{tabularx}
\usepackage[normalem]{ulem}
\usepackage{lscape}
\usepackage{placeins}
\usepackage{enumitem}
\usepackage{breakurl}

\expandafter\def\expandafter\UrlBreaks\expandafter{\UrlBreaks
 \do\a\do\b\do\c\do\d\do\e\do\f\do\g\do\h\do\i\do\j%
 \do\k\do\l\do\m\do\n\do\o\do\p\do\q\do\r\do\s\do\t%
  \do\u\do\v\do\w\do\x\do\y\do\z\do\A\do\B\do\C\do\D%
  \do\E\do\F\do\G\do\H\do\I\do\J\do\K\do\L\do\M\do\N%
  \do\O\do\P\do\Q\do\R\do\S\do\T\do\U\do\V\do\W\do\X%
  \do\Y\do\Z\do\*\do\-\do\~\do\'\do\"\do\-}%

\newtheorem{theorem}{Theorem}
\newtheorem{proposition}[theorem]{Proposition}

\newtheorem{cor}[theorem]{Corollary}

\theoremstyle{definition}
\newtheorem{definition}[theorem]{Definition}
\newtheorem{example}[theorem]{Example}
\newtheorem{remark}[theorem]{Remark}
\newtheorem{assumption}[theorem]{Assumption}

\newcolumntype{M}[1]{>{\centering\arraybackslash}m{#1}}
\newcolumntype{P}[1]{>{\centering\arraybackslash}p{#1}}

\renewcommand{\AA}{\mathcal{A}}
\renewcommand{\SS}{\mathcal{S}}
\newcommand{\OO}{\mathcal{O}}

\definecolor{NiceBlue}{rgb}{0.2,0.2,0.75}
\newcommand{\struc}[1]{{{\color{NiceBlue} #1}}}
\newcommand{\struct}[1]{{\emph{\color{NiceBlue} #1}}}

\tikzset{
  saveuse path/.code 2 args={
    \pgfkeysalso{#1/.style={insert path={#2}}}%
    \global\expandafter\let\csname pgfk@\pgfkeyscurrentpath/.@cmd\expandafter\endcsname
                           \csname pgfk@\pgfkeyscurrentpath/.@cmd\endcsname
    \pgfkeysalso{#1}},
  /pgf/math set seed/.code=\pgfmathsetseed{#1}}

\definecolor{MPIturquoise}{RGB}{66, 184, 178} 

\def\endexa{\hfill$\hexagon$}

\title{Algebraic Optimization of Sequential Decision Problems}

\author[M. Dressler]{Mareike Dressler}
\address{Mareike Dressler, School of Mathematics and Statistics, University of New South Wales, Sydney, NSW 2052, Australia.}

\author[M. Garrote-L\'{o}pez]{Marina Garrote-L\'{o}pez}
\address{Marina Garrote-L\'{o}pez, Department of Mathematics, University of British Columbia, Vancouver, V6T 1Z2 BC, Canada.}
 
\author[G. Mont\'{u}far]{Guido Mont\'{u}far}
\address{Guido Mont\'{u}far, Departments of Mathematics and Statistics, University of California, Los Angeles, 90095 CA, USA.}

\author[J. M\"{u}ller]{Johannes M\"{u}ller}
\author[K. Rose]{Kemal Rose}
\address{Guido Mont\'{u}far, Johannes M\"{u}ller, Kemal Rose, Max Planck Institute for Mathematics in the Sciences, 04103 Leipzig, Germany}

\newcommand{\sep}{, }

\begin{document}

\begin{abstract}
We study the optimization of the expected long-term reward in finite partially observable Markov decision processes over the set of stationary stochastic policies. In the case of deterministic observations, also known as state aggregation, the problem is equivalent to optimizing a linear objective subject to quadratic constraints. We characterize the feasible set of this problem as the intersection of a product of affine varieties of rank one matrices and a polytope. Based on this description, we obtain bounds on the number of critical points of the optimization problem. Finally, we conduct experiments in which we solve the KKT equations or the Lagrange equations over different boundary components of the feasible set, and compare the result to the theoretical bounds and to other constrained optimization methods. 
\end{abstract}

\maketitle

\section{Introduction}

Solving sequential decision problems has a long-standing history in
computer science, economics, mathematics, and statistics \cite{bellman1957markovian,howard1960dynamic,chernoff1968optimal}. 
Such problems include the optimal control of robots, machine maintenance, search problems, and inventory problems, which can be formulated in continuous or discrete time, space, and control variables~\cite{white1988further, bellman1966dynamic}. 
A sequential decision problem is particularly challenging if only partial information about the true state of the system is available to the acting agent. 

Partially observable Markov decision processes (POMDPs) offer a model for sequential decision-making under state uncertainty. 
Here, at every time step the agent selects an action and receives an instantaneous reward depending on the selected action and the current state, which in turn influence the state at the next time step. 
However, the agent selects its actions based on observations that might not fully reveal the underlying state. 
We study stochastic action selection mechanisms that do not depend on the prior history of observations but only on the current observation, which are known as memoryless, stationary, or reactive policies. 

A common measure for the performance of a policy is the expectation of the instantaneous rewards accumulated over time and discounted into the future. We will refer to this measure simply as the reward function. 
Identifying a policy that maximizes the reward is a challenging task since it is a nonconcave function that can exhibit non global strict local optima \cite{bhandari2019global}. 
Indeed it has been shown that this optimization problem is NP-hard in general~\cite{vlassis2012computational}. 
A common approach are local optimization procedures, such as policy gradient methods~\cite{sutton1999policy,azizzadenesheli2018policy}. Whereas global optimality guarantees for gradient methods in fully observable systems have been given in~\cite{bhandari2019global}, for general POMDPs we do not have such guarantees. 

Various approaches have been suggested to study the geometry underlying the optimization problem. A classic line of works has established that in the fully observable case (where the observation fully identifies the underlying state), the optimization problem is equivalent to a linear program over a polytope of feasible state-action frequencies~\cite{derman1970finite, kallenberg1994survey}. 
These studies have been complemented by the characterization of the set of feasible value functions of a Markov decision process as a finite union of polytopes~\cite{dadashi2019value, wang2022geometry, wu2022geometric}. 
However, for partially observable systems the geometry of the reward optimization problem is more complex. 
The problem can be formulated as a quadratically constrained linear program with the policy and the value function as search variables~\cite{amato2006solving}. 
More recently, the set of feasible state-action frequencies was described as a union of convex sets in~\cite{montufar2015geometry} and as a semialgebraic set in~\cite{mueller2021geometry}, who also provided a method for computing the polynomial constraints. 
This yields a polynomially constrained linear program with the state-action frequencies as search variables. 
The possible advantages of taking this constrained optimization perspective in state-action space were recently studied in~\cite{mueller2022rosa} using interior point methods. 

Related approaches have been proposed in other settings as well. In continuous time and space, a convex relaxation of linear quadratic control problems based on state-action frequencies has been proposed and studied in~\cite{lasserre2008nonlinear}. 
In~\cite{neyman2003real} the graphs of different stochastic games are described as semialgebraic sets, where (generalized) Nash equilibria including a convex relaxation for their computation have been studied with algebraic tools in~\cite{nie2021convex, portakal2022geometry}. 

\smallskip
In this paper we study finite POMDPs and build upon the recent work~\cite{mueller2021geometry}, which expresses reward optimization in POMDPs with memoryless stochastic policies as a linear program with polynomial constraints; that is, we are concerned with the optimization of a linear function over a nonconvex semialgebraic set. 
We focus on the case of deterministic observations, where, as we will see, the polynomial constraints are quadratic and can be written as a sum of certain $2\times2$ minors (Theorem~\ref{prop: feasible frequencies}). 
By investigating the geometry of the semialgebraic set, we determine upper bounds on the number of (complex) critical points of the reward optimization problem, i.e., its algebraic degree (Theorem~\ref{theo: algebraic_degree_of_subproblem}). 
We then provide a computational method that solves the optimization problem by computing the critical points via the Karush-Kuhn-Tucker conditions, whereby we identify ways to reduce the combinatorial complexity of the problem by focusing on relevant boundary components (Theorem~\ref{theo: location_of_maximizers}, \cite{montufar2017geometry}). 
We implement this approach using numerical algebra methods~\cite{HomotopyContinuation} that automatically certify the correctness of the results \cite{breiding2021certifying}. 
We use a convex relaxation of the polynomial problem to certify the global optimality of the results. 
Moreover, we observe that in specific instances this numerical algebraic approach leads to superior results than two commonly used optimization methods. Finally, we compare the number of critical points obtained in numerical experiments with our theoretical bounds. 

\smallskip

The paper is organized as follows. 
In Section~\ref{sec: POMDP}, we introduce partially observable Markov decision processes and related notation. 
In Section~\ref{sec: geometry of reward optimization}, we describe 
the geometry of the feasible set and its defining (in)equalities for the reward optimization problem in POMDPs with deterministic observations. 
In Section~\ref{sec: complexity of the problem}, we provide an upper bound on the number of critical points for the problem we are considering. 
Finally, in Section~\ref{sec: optimizing decision rules}, we use the description of reward maximization as a constrained polynomial optimization problem to numerically solve the critical equations. 

\medskip
\textbf{Notation:}
For a finite set $\mathcal X$ we denote the free linear space over $\mathcal X$ by $\struc{\mathbb R^\mathcal X} = \{f\colon\mathcal X\to\mathbb R\}$ and the simplex of probability distributions over \(\mathcal X\) as $\struc{\Delta_{\mathcal X}} = \left\{\mu\in \mathbb R^{\mathcal X} : \sum_x \mu_x = 1 \text{ and } \mu \geq 0 \right\}$.
The \struct{conditional probability polytope} consisting of all column-stochastic matrices\footnote{We choose to work with column-stochastic rather than row-stochastic matrices to have $Q_{yx}=Q(y|x)$, which makes composition of two Markov kernels $Q_1\circ Q_2$ equivalent to matrix multiplication $Q_1Q_2$.} 
in $\mathbb{R}^{\mathcal Y\times\mathcal X}$ is the product $\struc{\Delta_\mathcal Y^\mathcal X}= \Delta_\mathcal Y\times\dots\times\Delta_\mathcal Y$. 
We call the elements of this set \struct{conditional probability distributions} or \struct{Markov kernels} from $\mathcal X$ to $\mathcal{Y}$. Given a Markov kernel \(Q\in \Delta_{\mathcal Y}^{\mathcal X},\) 
the conditional probability $Q(y|x)$ is the entry $Q_{yx}$. 
Note that a composition of Markov kernels is matrix multiplication.
For a probability distribution $p\in\Delta_{\mathcal{X}}$ and a Markov kernel $Q\in\Delta^{\mathcal{X}}_{\mathcal{Y}}$ we denote their composition into a joint probability distribution by $\struc{p\ast Q} \in \Delta_{\mathcal{X}\times\mathcal{Y}}$ and define it as 
$p\ast Q = \operatorname{diag}(p)Q^T$, that is, with entries $(p\ast Q)(x, y)\coloneqq p(x)Q(y|x)$. 
For a subset $A\subseteq\mathcal X$ we denote the complement $\mathcal X\setminus A$ of $A$ in $\mathcal X$ by $A^c$. 

\section{Partially observable Markov decision processes}
\label{sec: POMDP}

Partially observable Markov decision processes provide a powerful model to describe sequential decision making problems with state uncertainty. 

\begin{definition}
A finite \struct{partially observable Markov decision process} or shortly \struct{POMDP} is a tuple \(\struc{(\mathcal S, \mathcal O, \mathcal A, \alpha, \beta, r)}\), where \(\mathcal S, \mathcal O\), and \(\mathcal A\) are finite sets called the \struct{state}, \struct{observation}, and \struct{action space} respectively and \(\alpha\in\Delta_{\mathcal S}^{\mathcal S\times\mathcal A}\) and \(\beta\in\Delta_{\mathcal O}^{\mathcal S}\) are Markov kernels, which we call the \struct{transition} and \struct{observation kernel} respectively. Furthermore, we consider an \struct{instantaneous reward vector} \(r\in\mathbb R^{\mathcal S\times\mathcal A}\). We denote the cardinalities of $\SS, \AA$, and $\OO$ by $\struc{n_\SS}, \struc{n_\AA}$, and $\struc{n_\OO}$. 
\end{definition}

From a modeling perspective, $\alpha(s'|s, a)$ is the probability of transitioning from state $s$ to state $s'$ upon taking action $a$, and $\beta(o|s)$ is the probability of making the observation $o$ if the system is in state $s$. 
The entry $r_{sa}$ corresponds to an instantaneous reward received upon selecting action $a$ in state $s$. 

A \struct{(memoryless stochastic) policy} is a column-stochastic matrix \(\struc{\pi}\in\Delta_{\mathcal A}^{\mathcal O}\) from the set of observations to the set of actions. 
The entry $\pi(a|o)$ is the probability with which action $a\in\AA$ is selected given the observation $o\in\OO$. 
A policy can be interpreted as a randomized decision rule that encodes which action should be taken, based on the current observation. 
Every policy $\pi\in\Delta_\mathcal A^\mathcal O$ defines a transition kernel \(\struc{P_{\pi \circ \beta}} \in\Delta_{\mathcal S\times\mathcal A}^{\mathcal S\times\mathcal A}\) with entries 
 \[ 
 P_{\pi \circ \beta}(s^\prime, a^\prime|s, a) \coloneqq \alpha(s^\prime|s, a) (\pi\circ\beta)(a^\prime|s^\prime) ,
 \]
which we call the \struct{state-action transition kernel} associated with $\pi \circ \beta$ and $\alpha$. 
Given an initial distribution, the state-action transition kernel $P_{\pi\circ \beta}$ defines a Markov process on the state-action space $\SS\times\AA$. 

One is particularly interested in the probability that the Markov process assigns to any given state-action pair, averaged over time, whereby it is convenient to discount events at larger times $t$ by weighting them by $(1-\gamma) \gamma^t$ for a \struct{discount factor} $\struc{\gamma}\in(0,1)$. 
Given an initial state distribution \(\mu\in\Delta_{\mathcal S}\) and a discount factor $\gamma\in (0, 1)$, one thus defines the \struct{(discounted) state-action frequency} associated with policy $\pi\in\Delta^\OO_\AA$ as 
$$
\struc{\eta^\pi} \coloneqq (1-\gamma)\sum_{t\geq 0} \gamma^t P_{\pi\circ\beta}^t(\mu\ast(\pi\circ\beta)) = (1-\gamma)(I-\gamma P_{\pi \circ \beta} )^{-1} (\mu\ast(\pi\circ\beta)), 
$$
where $I$ is the identity matrix; see~\cite{derman1970finite, kallenberg1994survey}.
We further define the map 
\begin{equation}\label{eq:Phi}
 \begin{array}{rccl}
\struc{\Phi}\colon & \Delta_\mathcal A^\mathcal O & \rightarrow & \Delta_{{\mathcal S} \times {\mathcal A}} \\
 & \pi & \mapsto & \eta^\pi = (1-\gamma)(I-\gamma P_{\pi \circ \beta} )^{-1} (\mu\ast(\pi\circ\beta)). 
 \end{array}
\end{equation}
Elementary calculations show $\eta^\pi(a|s) = (\pi\circ\beta)(a|s)$. 
We denote the state-marginal of $\eta^\pi$ by $\struc{\rho^{\pi}_s} = \sum_{a\in\AA}\eta^\pi_{sa}$ and refer to it as the \struct{state frequency}. 
By definition of conditional probability distributions it holds that 
\begin{equation}
\label{eq:conditional_probability}
 \eta^\pi_{sa} = \eta^\pi(a|s)\rho^\pi_s = (\pi\circ\beta)(a|s)\rho^\pi_s . 
\end{equation}
Finally, as a measure for the performance of policies, we introduce the \struct{reward function}\footnote{
More precisely, this is the infinite-horizon expected discounted reward function.}:
\begin{equation}
    \label{eq:infinite_horizon_reward}
\struc{R(\pi)} \coloneqq \sum_{s\in\SS,a\in\AA} r_{sa} \Phi(\pi)_{sa} = \langle r, \Phi(\pi)\rangle_{\SS\times\AA}.
\end{equation}
The reward function $R$ is a widely used criterion to evaluate the performance of a policy. It is equal to the expected value 
$\mathbb E\left[(1-\gamma)\sum_{t\ge0} \gamma^t r(s_t,a_t)\right]$
of the (discounted) accumulated instantaneous rewards along state-action trajectories distributed according to the Markov process with transition kernel $P_{\pi \circ \beta}$ and initial state-action distribution $\mu\ast (\pi\circ\beta)$. 
We refer to standard textbooks for an in-depth discussion \cite{howard1960dynamic, derman1970finite, puterman2014markov}. 

We consider the following \struct{reward 
optimization problem (ROP)}, which is the standard problem in (discounted) Markov decision processes: 
\begin{equation}\label{eq:rewardMaximization}
\tag{ROP}
 \operatorname{maximize} \; R(\pi) \quad \text{subject to } \pi\in\Delta_\mathcal A^\mathcal O.
\end{equation}

In this work, we focus on \textbf{deterministic observations} $\beta\in\Delta_\OO^\SS\cap\{0,1\}^{\OO\times\SS}$, where we can identify the observation kernel with a deterministic mapping $\struc{g_\beta}\colon\SS\to\OO$. 
We denote the fibers of $g_\beta$ by $\struc{S_o}\coloneqq \{s\in\mathcal S : g_\beta(s) = o\}$ and their cardinality by $\struc{d_o} \coloneqq \lvert S_o\rvert$. 
Note that the fibers $S_o$ are a disjoint partition of the states $\mathcal S$ and hence $(d_o)_{o\in\OO}$ is a partition of $n_\SS$, i.e., $\sum_{o\in\OO} d_o = n_\SS$. 
This special type of partial observability is known in the literature as state-aggregation. 

\begin{example}\label{ex:transMech_deterministic}
Let \((\mathcal S, \mathcal O, \mathcal A, \alpha, \beta, r)\) be a POMDP with state space $\SS = \{s_1, s_2,s_3\}$, action space $\AA = \{a_1,a_2\}$, and observation space $\OO=\{o_1,o_2\}$. 
We consider the (deterministic) transitions depicted in Figure~\ref{fig:transMech_deterministic}, which correspond to the column stochastic matrix 
$$ 
\alpha = 
\bordermatrix{ & \textcolor{gray}{ s_1,a_1} & \textcolor{gray}{s_1, a_2} & \textcolor{gray}{s_2, a_1} & \textcolor{gray}{s_2, a_2} & \textcolor{gray}{s_3, a_1} & \textcolor{gray}{s_3, a_2} \cr
 \textcolor{gray}{s_1} & 1 & 0 & 0 & 1 & 0 & 0 \cr
 \textcolor{gray}{s_2} & 0 & 0 & 0 & 0 & 0 & 1 \cr
 \textcolor{gray}{s_3} & 0 & 1 & 1 & 0 & 1 & 0} \in\Delta_\SS^{\SS\times\AA} . 
$$ 
Assume the agent cannot distinguish the states $s_1$ and $s_2$, so that the observation kernel is 
$$ 
\beta = 
\bordermatrix{ & \textcolor{gray}{s_1} & \textcolor{gray}{s_ 2} & \textcolor{gray}{s_3} \cr
 \textcolor{gray}{o_1} & 1 & 1 & 0 \cr
 \textcolor{gray}{o_2} & 0 & 0 & 1}\in\Delta_\OO^\SS. 
$$
Hence, we are optimizing over stochastic matrices 
\[ 
\pi = 
\bordermatrix{ & \textcolor{gray}{o_1} & \textcolor{gray}{o_ 2} \cr
 \textcolor{gray}{a_1} & \pi_{a_1o_1} & \pi_{a_1o_2} \cr
 \textcolor{gray}{a_2} & \pi_{a_2o_1} & \pi_{a_2o_2}}\in\Delta_\AA^\OO. 
 \]
Further, we consider a uniform initial distribution $\mu \in\Delta_\SS$ and a discount factor $\gamma=1/2$. 
\begin{figure} 
 \centering
 \begin{tikzpicture}
 \node[shape=rectangle, draw=black, 
 minimum size=1cm] (A) at (0,0) {\(s_1\)};
 \node[shape=rectangle,draw=black,minimum size=1cm] (B) at (3,0) {\(s_2\)};
 \node[shape=circle,draw=black,minimum size=1cm] (C) at (6,0) {\(s_3\)};
 \path [->] (B) [bend right=-20] edge node[above] {\(a_1\)} (C);
 \path [->] (B) [bend left=20] edge node[below] {\(a_2\)} (A);
 \path [->] (C) [bend left=20] edge node[below] {\(a_2\)} (B);
 \path [->] (A) [bend right=-40] edge node[above] {\(a_2\)} (C);
 \path [->] (A) [loop left] edge node[left] {\(a_1\)} (A);
 \path [->] (C) [loop right] edge node[right] {\(a_1\)} (C);
\end{tikzpicture}
\captionsetup{width=.85\linewidth}
\caption{\small{Transition graph of Example~\ref{ex:transMech_deterministic}; states $s_1, s_2$ lead to observation $o_1$, and $s_3$ leads to observation $o_2$. }}
\label{fig:transMech_deterministic}
\end{figure} 
Finally, let us assume the instantaneous reward vector is $r(s,a) = \delta_{s_1s}$, 
which corresponds to a reward of $+1$ obtained in state $s_1$. 
Combining the Neumann series with Cramer's rule (see~\cite{mueller2021geometry}) one sees that the reward function $R$ is a rational function with the explicit expression $R(\pi) =  \frac{f(\pi)}{2g(\pi)} - \frac12
$, where $f$ and $g$ are determinantal polynomials given by 
\begin{align}\label{eq:rationalExpression}
 \begin{split}
f(\pi) & = \det\begin{pmatrix}
-0.49 \pi_{a_1o_1} + 0.5 \pi_{a_2o_1} + 1 & -0.99 \pi_{a_2o_2} \\ -0.5 \pi_{a_1o_1} - 0.49 \pi_{a_2o_1} & -0.99 \pi_{a_1o_2} + 1
\end{pmatrix}
\\ & = \ \scriptstyle \pi_{a_1o_1}^2\pi_{a_2o_2} -2\pi_{a_1o_1}\pi_{a_2o_1}\pi_{a_1o_2} - 2\pi_{a_2o_1}^2\pi_{a_1o_2} -\pi_{a_2o_1}^2 \pi_{a_2o_2} + 4\pi_{a_1o_1}\pi_{a_2o_1} \\
  & \quad \scriptstyle + 2 \pi_{a_1o_1}\pi_{a_1o_2} - 6 \pi_{a_1o_1}\pi_{a_2o_2} + 4\pi_{a_2o_1}^2 - 4 \pi_{a_2o_1}\pi_{a_1o_2} - 4 \pi_{a_1o_1} + 8 \pi_{a_2o_1} - 12 \pi_{a_1o_2} + 24
\end{split}
\end{align}
and 
\begin{align}
\begin{split}
    g(\pi) & = \det\begin{pmatrix} -0.99 \pi_{a_1o_1} + 1 & -0.99 \pi_{a_2o_2} \\ - 0.99 \pi_{a_2o_1} & -0.99 \pi_{a_1o_2} + 1 \end{pmatrix} \\
    & = \ \scriptstyle 3 \pi_{a_1o_1}^2 \pi_{a_2o_2} - 3 \pi_{a_2o_1}^2\pi_{a_2o_2} + 6 \pi_{a_1o_1}\pi_{a_1o_2} - 6\pi_{a_1o_1}\pi_{a_2o_2} -12 \pi_{a_1o_1} - 12 \pi_{a_1o_2} + 24.
\end{split}
\end{align}
The reward function is to be optimized over the observation policy, that is, we have
\begin{align}\label{eq:exOriginalOpt}
\operatorname{maximize} \; R(\pi) 
\quad \text{subject to }
\left\{
\begin{array}{ll}
 \pi_{oa}\geq 0 & \text{for all } o\in\mathcal{O}, a\in \mathcal{A},\\
 \sum_{a\in \mathcal{A}}\pi_{oa}=1 &  \text{for all } o\in\mathcal{O}.\\
\end{array}
\right.
\end{align}
\endexa
\end{example}

\section{The geometry of reward optimization}
\label{sec: geometry of reward optimization}
In this section, we discuss the formulation of the reward optimization problem as a polynomially constrained linear program from~\cite{mueller2021geometry}. 
For deterministic observations we provide a new description of the feasible state-action frequencies as the intersection of a product of varieties of rank-one matrices, an affine space, and the simplex (see Theorem~\ref{prop: feasible frequencies}). 

\medskip

Clearly, optimizing $R(\pi) = \langle r , \Phi(\pi) \rangle$ over 
$\Delta_\mathcal A^\mathcal O$
is equivalent to the \struct{reward maximization problem in the state-action space (ROPSA)}: 
\begin{equation}\label{eq:rewardOptimizationImplicit}
\tag{ROPSA}
 \operatorname{maximize} \;\langle r , \eta \rangle \quad \text{subject to } \eta\in\Phi(\Delta_\mathcal A^\mathcal O).
\end{equation}
By definition, the feasible set $\Phi(\Delta_\mathcal A^\OO)$ is a subset of the probability simplex $\Delta_{\SS\times\AA}$. 
Cramer's rule implies that the parametrization $\Phi$ is a rational map and hence, by the Tarski-Seidenberg theorem, the range $\Phi(\Delta_\AA^\OO)$ is semialgebraic. 
Next we discuss the solution of the implicitization problem for the parametric set $\Phi(\Delta_\AA^\OO)$ as recently given by~\cite{mueller2021geometry}, i.e., a representation of this set as the solution set to a list of polynomial (in)equalities. 

The mapping $\Phi$ can be seen as a composition $\Psi\circ f_\beta$ of a linear and non-linear map, 
illustrated in Figure~\ref{fig:range4} for the POMDP of Example~\ref{ex:transMech_deterministic}, with 
\begin{align*}
 \begin{array}{r l}
 \struc{f_\beta}\colon \Delta_\mathcal A^\mathcal O & 
 \longrightarrow \Delta_\mathcal A^\mathcal S \\
 \pi & \longmapsto \tau = \pi\circ\beta 
 \end{array} \quad \text{and}\quad
 \begin{array}{r l}
 \struc{\Psi} \colon \Delta_\mathcal A^\mathcal S & \longrightarrow \Delta_{\mathcal S\times\mathcal A} \\
 \tau & \longmapsto \eta = (1-\gamma)(I-\gamma P_\tau)^{-1} (\mu\ast \tau). 
 \end{array}
\end{align*}
\begin{figure}[]
 \centering
 \begin{tikzpicture}
 \node at (0,-.2){ 
 \begin{tikzpicture}[scale=2,axis/.style={->,dashed},thick] 
 
 \fill[
 color=MPIturquoise] (0,0) rectangle (1, 1);
 \draw (0,0) rectangle (1,1);
 \draw[axis] (0,0)--(1.5,0) node[below]{\tiny $\pi(a_1|o_1)$};
 \draw[axis] (0,0)--(0,1.5) node[right]{\tiny $\pi(a_1|o_2)$};
 \node at (.5,.5) {$\Delta^\mathcal{O}_{\mathcal{A}}$}; 
 \fill (.75,.75) circle (.5pt) node[right] {$\pi$};
 \end{tikzpicture}
 };
 \node at (5,0){ 
 \begin{tikzpicture}[scale=2, 
 axis/.style={->,dashed},thick]
 \draw[axis] (1, 0, 0) -- (1.5, 0, 0) node [below] {\tiny $\tau(a_1|s_1)$};
 \draw[axis] (0, 1, 0) -- (0, 1.5, 0) node [right] {\tiny $\tau(a_1|s_3)$};
 \draw[axis] (0, 0, 1) -- (0, 0, 1.8) node [right] {\tiny $\tau(a_1|s_2)$};
 
 \coordinate (d1) at (0,0,0){};
 \coordinate (d2) at (0,0,1){};
 \coordinate (d3) at (0,1,0){};
 \coordinate (d4) at (1,0,0){};
 \coordinate (d5) at (0,1,1){};
 \coordinate (d6) at (1,0,1){};
 \coordinate (d7) at (1,1,0){};
 \coordinate (d8) at (1,1,1){};
 
 \draw [gray,dashed] (d2)--(d1)--(d3);
 \draw [gray,dashed] (d1)--(d4);
 \fill[color=MPIturquoise] (d1)--(d3)--(d8)--(d6)--cycle; 
 \draw [] (d5)--(d3)--(d7);
 \draw [] (d6)--(d8)--(d5)--(d2)--(d6)--(d4)--(d7)--(d8);
 \draw [thin] (d1)--(d3)--(d8)--(d6)--(d1);
 
 \node at (.5,.25,.5) {$\scriptstyle f_\beta(\Delta^\mathcal{O}_{\mathcal{A}})$}; 
 \node at (1.15,1.15,0.) {$ \Delta^\mathcal{S}_{\mathcal{A}}$}; 
 \fill (.75,.75, .75) circle (.5pt) node[left] {$\tau$};
 \end{tikzpicture}

 }; 
 \node at (10.7,-.2) { 
 \begin{tikzpicture}[scale=1.8]
 \node at (0,0) {\includegraphics[width=0.25\textwidth]{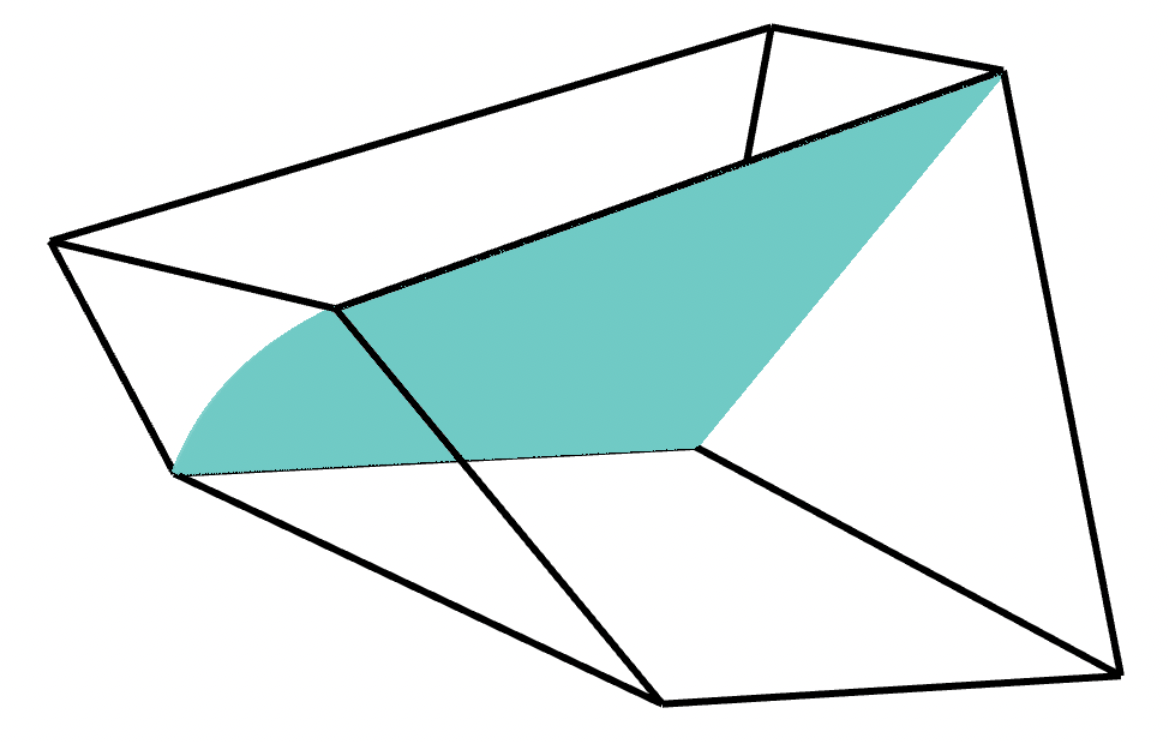}};
 \node (Nt) at (1.5,.05)  {\textcolor{MPIturquoise}{$ \Phi(\Delta^\mathcal{O}_{\mathcal{A}})$}}; 
 \node at (-.3,.75) {$ \Psi(\Delta^\mathcal{S}_{\mathcal{A}})$}; 
 \fill (-.3,.05) circle (.5pt) node[right] {$\eta$};
 \coordinate (N) at (.2,.2);
 \draw[thin] (N) -- (Nt);
 \end{tikzpicture}
 };
 \node at (2,0) {$\xrightarrow[\text{linear}]{f_\beta}$}; 
 \node at (7,0) {$\xleftrightarrow[\text{rational}]{\Psi}$};
\end{tikzpicture}
\captionsetup{width=.8\linewidth}
\caption{\small{Shown is the policy polytope $\Delta_\AA^\OO$ (left), the effective policy polytope $f_\beta(\Delta_\AA^\OO)$ within the state policy polytope $\Delta_\AA^\SS$ (middle), and the set of feasible state-action frequencies $\Phi(\Delta_\OO^\AA)$ within the state-action polytope $\Psi(\Delta^\SS_\AA)$ (right). 
Note that $\Phi(\Delta_\OO^\AA)$ is the nonlinear solution set of the constraints given in equation~\eqref{eq:exPolynomialProgramDet} 
and $\Psi(\Delta^\SS_\AA)$ is a three-dimensional polytope that is combinatorially equivalent to the cube. }} 
\label{fig:range4}
\end{figure}

We recall the following classic result. 
\begin{proposition}[The state-action polytope of Markov decision processes, \cite{derman1970finite}]
\label{prop: state_action_polytope_fully_observable}
The image $\Psi(\Delta_\AA^\SS)$ is a polytope given by $\Psi(\Delta_\AA^\SS) = \mathcal L\cap\Delta_{\SS\times\AA}$, where
\begin{equation}\label{eq:linearSpace}
 \struc{\mathcal L} \coloneqq \left\{ \eta\in\mathbb R^{\SS\times\AA} :\ \ell_s(\eta) = 0 \text{ for all } s\in\SS \right\},
\end{equation}
and $ \struc{\ell_s(\eta)} \coloneqq \sum_{a} \eta_{sa} - \gamma\sum_{s',a'} \eta_{s'a'}\alpha(s|s', a') - (1-\gamma) \mu_s$.
\end{proposition}

In particular, the set of state-action frequencies $\Psi(\Delta_\AA^\SS)$ of a fully observable Markov decision process forms a polytope, referred to as the \struct{state-action polytope}. 
The constraints encoded in $\mathcal L$ describe a generalized stationarity property of the state-action frequencies, recovering stationarity in the limit where the discount factor is $\gamma=1$. 
In order to relate the space of state policies to state-action frequencies we make the following assumption.
\begin{assumption}[Positivity]\label{ass:positivity}
For every $s\in\mathcal S$ and $\pi\in\Delta_\AA^\OO$, we assume that $\sum_{a} \eta_{sa}>0$. 
\end{assumption}
This assumption is satisfied, for example, if the system is ergodic or if the initial distribution $\mu\in\Delta_\SS$ has full support, i.e., has only strictly positive entries. This can be seen by interpreting $\sum_{a}\eta_{sa}$ as a weighted average of the time spent in state $s$ when following the policy $\pi$. 
An important consequence of this assumption is that the state policies $\tau$ and the state-action frequencies $\eta$ are in one-to-one correspondence, whereby the state policies can easily be computed from the state-action frequencies by conditioning.

\begin{proposition}[\cite{mueller2021geometry}]\label{prop:birational}
Under Assumption~\ref{ass:positivity}, the mapping $\Psi\colon\Delta_\AA^\SS\to\Psi(\Delta_\AA^\SS)$ is rational and bijective with rational inverse given by conditioning
\begin{align*}
 \struc{\Gamma}\colon 
\Psi(\Delta_\AA^\SS) &\longrightarrow \Delta_\AA^\SS
 \\
 \eta &\longmapsto \tau, \quad \text{where $\tau_{as} = \frac{\eta_{sa}}{\sum_{a'}\eta_{sa'}}$}. 
\end{align*}
\end{proposition}
The function $\Psi$ is defined everywhere on $\Delta^\mathcal{S}_\mathcal{A}$ and bijectively identifies the defining inequalities of the polytope $f_\beta(\Delta_\mathcal A^\mathcal O)$ within $\Delta_\AA^\SS$
with the defining inequalities of
$\Phi(\Delta^\mathcal{O}_\mathcal{A})$ within $\Psi(\Delta_\AA^\SS)$
via the pullback along $\Gamma$. 
This relates the geometry of $\Phi(\Delta_\AA^\OO)$ and $f_\beta(\Delta_\AA^\OO)$. The defining inequalities of $f_\beta(\Delta_\AA^\OO)$ can be computed algorithmically, see e.g.,~\cite{Jones:169768}. 
As we demonstrate in what follows, for deterministic observations the defining inequalities can be given in closed form.
In particular, we show the following:

\begin{theorem}[Feasible state-action frequencies]
\label{prop: feasible frequencies}
For deterministic observation $\beta$ the set of feasible state-action frequencies $\Phi(\Delta_\AA^\OO)$ is the intersection $\mathcal L\cap\mathcal X \cap\Delta_{\SS\times\AA}$ of the linear space $\mathcal{L}$ defined in~\eqref{eq:linearSpace}, the product of real determinantal varieties 
\begin{align*}
\struc{\mathcal X} \coloneqq \Big \{ \eta \in \mathbb{R}^{\SS \times A} : \ \eta_{sa}\eta_{s'a'} - \eta_{sa'} \eta_{s'a} = 0 \ \forall a, a'\in\AA \text{ and } s, s'\in \SS \ \text{with}\  g_\beta(s) = g_\beta(s')
 \Big \},
\end{align*}
and the probability simplex $\Delta_{\SS\times\AA}$. 
In particular, the only inequalities are of the form $\eta\geq0$.
\end{theorem}

We call $\mathcal L\cap\mathcal X$ the \struct{state-aggregation variety}.
Note that $\mathcal{X}$ is determined by the condition that for every observation $o$ the $d_o\times n_\AA$ submatrix $ \left( \eta_{sa} \right)_{ s \in S_o,\; a \in \AA }$ of $\eta$, consisting of all entries $\eta_{sa}$ with $\beta(s) = o$,
has rank one.
In particular,
the projective variety associated to $\mathcal X$ is a join of Segre varieties.
 
\begin{proof}[Proof of Theorem~\ref{prop: feasible frequencies}]
Proposition \ref{prop: state_action_polytope_fully_observable} provides a description of the polytope
$\Psi(\Delta_\AA^\SS)$
as the intersection 
$\mathcal L\cap \Delta_{\SS\times\AA}$
so we are left with finding defining equations for $\Psi(f_\beta(\Delta_\AA^\OO))$ 
in $\Psi(\Delta_\AA^\SS)$. 
To do this, observe that the polytope $f_\beta(\Delta_\AA^\OO)$ consists of those elements
$\tau\in\Delta_\AA^\SS$ satisfying the linear equations
$
\tau_{as} - \tau_{as'} 
$
for all $a \in \AA$, $s,s' \in S$ such that $g_\beta(s)=g_\beta(s')$.
In other words, all columns of $\tau$ indexed by states with equal observations coincide. 
Fix an action $a_o\in\AA$ and a state $s_o\in S_o$ for each observation $o\in\OO$. 
Then the non-redundant defining equalities of $f_\beta(\Delta_\AA^\OO)$ are given by 
\[ 
l^o_{sa}(\tau) \coloneqq \tau_{as} - \tau_{as_o} = 0 , 
\]
for all observations $o\in\OO$, actions $a\in\AA\setminus\{a_o\}$, and states in the fiber $s\in S_o\setminus\{s_o\}$. 
These equations determine the range of $f_\beta$ as a function $\mathbb{R}^{\AA\times\OO}\to\mathbb{R}^{\AA\times\SS}$ (corresponding to the set $\mathcal{U}$ in \cite[Theorem 12]{mueller2021geometry}). 
After applying the pullback $\Gamma^{\ast}$ of the conditioning map $\Gamma$ to the linear functions $l_{sa}^o$ we get the rational equations
\begin{align}\label{eq: rational_cond_on_eta}
(\Gamma^{\ast}l^o_{sa})(\eta) = l^o_{sa}(\Gamma(\eta)) = \eta_{sa}\left(\sum_{a'\in \mathcal A} \eta_{sa'}\right)^{-1} - \eta_{s_oa}\left(\sum_{a'\in \mathcal A} \eta_{s_oa'}\right)^{-1} = 0, 
\end{align}
which we rephrase as the vanishing of the polynomials
\begin{align}\label{eq:definingPolynomialsDeterministic}
\struc{p^o_{sa}(\eta)} \coloneqq \eta_{sa} \sum_{a'\in \mathcal A} \eta_{s_oa'} - \eta_{s_oa} \sum_{a'\in \mathcal A} \eta_{sa'} = \sum_{a'\in\AA\setminus\{a\}}(\eta_{sa} \eta_{s_oa'} - \eta_{s_oa} \eta_{sa'}).
\end{align}
These are defining polynomial equations of $\Psi(f_\beta(\Delta_\AA^\OO))$ in $\Psi(\Delta_\AA^\SS)$.
Let now $\mathcal W$ be the variety determined by the equations $\eqref{eq: rational_cond_on_eta}$.
It remains to show $\mathcal X\cap\mathcal L\cap\Delta_{\SS\times\AA} = \mathcal W\cap\mathcal L\cap\Delta_{\SS\times\AA}$.
Since $ p^o_{sa}$ is a linear combination of $2\times 2$ minors, we have the inclusion $\mathcal{X} \subseteq \mathcal W$.
On the other hand, equation 
\eqref{eq: rational_cond_on_eta} implies the linear dependence of the two vectors
    \[ (\eta_{sa})_{a }, \ (\eta_{s_oa})_a \in \mathbb{R}^\AA \]
for every observation $o$ and state $s \in S_o$.
Consequently, every $2\times 2$ minor in the definition of $\mathcal{X}$ vanishes on $\mathcal W\cap\mathcal L\cap\Delta_{\SS\times\AA}$.
This shows the desired inclusion
    \[ 
    \mathcal W\cap\mathcal L\cap\Delta_{\SS\times\AA} \subseteq \mathcal X\cap\mathcal L\cap\Delta_{\SS\times\AA} , 
    \]
which finishes the proof. 
\end{proof}

Hence by Theorem~\ref{prop: feasible frequencies}, in the case of a deterministic observation kernel all inequalities are linear and the equalities are either linear or $2\times2$ principal minors. 
This is in contrast to the case of general observation kernels, where nonlinear defining inequalities appear and the polynomial constraints might be of higher degree (see \cite[Theorem 16]{mueller2021geometry}). 
Since all defining equalities of $\mathcal{X}$ are binomial, it is a toric variety. 
The following monomial parametrization of $\mathcal{X}$ can be inferred from the discussion of the family of state-frequencies and equation \eqref{eq:conditional_probability}: 
\begin{align*}
 \mathbb{R}^{\SS} \times \mathbb{R}^{\AA \times \OO} &\longrightarrow \mathcal{X} \\ (\rho, \pi) &\longmapsto \eta, \quad \text{where\; $\eta(s,a) = \pi(a|g_\beta(s)) \rho(s)$}. 
\end{align*}
The following characterization of the set of feasible state-action frequencies with fewer equations will be useful later. 

\begin{cor}[Alternative characterization of feasible state-action frequencies]\label{cor:Y} 
For deterministic observation $\beta$, fix an arbitrary action $a_o\in\AA$
and an arbitrary state $s_o\in S_o$ for every $o\in\OO$. 
The set of feasible state-action frequencies $\Phi(\Delta_{\mathcal A}^\mathcal O)$ can be described as the intersection $\mathcal L\cap\mathcal Y\cap\Delta_{\mathcal S\times\mathcal A}$, where 
\[ 
\struc{\mathcal Y} \coloneqq \left\{ \eta\in\mathbb R^{\mathcal S\times\mathcal A} : \ p^{o}_{sa}(\eta) = 0 \text{ for all } o\in\OO, a\in\AA\setminus\{a_o\}, s\in S_o\setminus\{s_o\} \right\}, 
\]
and the polynomials $p_{sa}^{o}$ are given in~\eqref{eq:definingPolynomialsDeterministic}. 
$\mathcal Y$ is a complete intersection of these polynomials. 
\end{cor}
\begin{proof}
This follows directly from the proof of Theorem~\ref{prop: feasible frequencies}.
\end{proof}

\begin{example}
\label{ex:transMech_deterministic2}
We continue Example~\ref{ex:transMech_deterministic} from above. 
The defining (in)equalities of $\Phi(\Delta_\AA^\OO)$, described in ~\eqref{eq:linearSpace} and Theorem~\ref{prop: feasible frequencies}, take the form 
\begin{align*}
 \mathcal L = \{\eta\in \mathbb R^{3\times 4} : \ &
 3\eta_{s_1a_1} + 6\eta_{s_1a_2} - 3\eta_{s_2a_2} - 1 =0, \ 6\eta_{s_2a_1} + 6\eta_{s_2a_2} - 3\eta_{s_3a_2} - 1 =0, \\ 
 & 3\eta_{s_3a_1} + 6\eta_{s_3a_2} - 3\eta_{s_1a_2} - 3\eta_{s_2a_1} - 1 =0
 \}
\end{align*}
and
    \[ \mathcal X = \{\eta \in \mathbb R^{3\times 4}: \ \eta_{s_1a_1}\eta_{s_2a_2} - \eta_{s_1a_2} \eta_{s_2a_1} = 0 \}. \]
Thus, the reward optimization problem~\eqref{eq:rewardOptimizationImplicit} is
\begin{align}\label{eq:exPolynomialProgramDet}
    \operatorname{maximize} \; \eta_{s_1a_1} + \eta_{s_1a_2} 
    \quad \text{subject to }
    \left\{
    \begin{array}{rl}
        3\eta_{s_1a_1} + 6\eta_{s_1a_2} - 3\eta_{s_2a_2} - 1 = 0\;\; \\
        6\eta_{s_2a_1} + 6\eta_{s_2a_2} - 3\eta_{s_3a_2} - 1 = 0\;\; \\
        3\eta_{s_3a_1} + 6\eta_{s_3a_2} - 3\eta_{s_1a_2} - 3\eta_{s_2a_1} - 1 = 0\;\; \\
        \eta_{s_1a_1}\eta_{s_2a_2} - \eta_{s_1a_2} \eta_{s_2a_1} = 0\;\; \\
        \eta_{s_1a_1}, \eta_{s_1a_2}, \eta_{s_2a_1}, \eta_{s_2a_2}, \eta_{s_3a_1}, \eta_{s_3a_2} \ge 0. 
\end{array}
\right.
\end{align}
The feasible set of this optimization problem is shown on the right in Figure~\ref{fig:range4}. 
Comparing this to the optimization problem over the policy polytope~\eqref{eq:exOriginalOpt} with objective function~\eqref{eq:rationalExpression}, 
now the constraints are more complex and nonlinear but the objective is linear. 
\endexa
\end{example}

\section{Combinatorial and algebraic complexity of the problem}
\label{sec: complexity of the problem}

In this section we study the number of critical points of the reward optimization problem in the case of deterministic observations. 
We apply methods from polynomial optimization and in particular the theory of algebraic degrees developed in~\cite{Nie2009AlgebraicDO} to obtain upper bounds on the number of complex critical points. 
A similar approach was pursued in~\cite{mueller2021geometry} for the case of invertible observation matrix $\beta$, in which case there are linear equations and polynomial inequalities. 

\medskip
The description of $\Phi(\Delta_\mathcal A^\mathcal O)$ obtained in Corollary~\ref{cor:Y}
allows to reformulate the reward optimization problem~\eqref{eq:rewardOptimizationImplicit} as the following constrained polynomial optimization problem:  
\begin{equation}\label{eq:polyFormRewMax}
\tag{POP}
 \operatorname{maximize} \;\langle r, \eta \rangle \quad \text{subject to } 
 \left\{
\begin{array}{rl}
 \ell_s(\eta) = 0 & \text{for } s\in\mathcal S,\\
 p_{sa}^o(\eta) = 0 & \text{for } o\in \OO, a\in\AA\setminus\{a_o\}, s\in S_o\setminus\{s_o\}, \\
 \eta_{sa} \ge 0 & \text{for } s\in\SS, a\in\AA,
\end{array}
\right.
\end{equation}
where the linear constraints $\ell_s$ are given in Proposition~\ref{prop: state_action_polytope_fully_observable}, the polynomial constraints $p_{sa}^o(\eta)$ are provided in~\eqref{eq:definingPolynomialsDeterministic} taking a fixed action $a_o\in\AA$ and a fixed state $s_o\in S_o$ for each observation $o\in\OO$, and the inequality constraints simply ensure the entries of $\eta$ being nonnegative. 
Observe that problem~\eqref{eq:polyFormRewMax} is in fact a 
quadratically constrained linear program.

We bound the number of critical points individually for each boundary component of the feasible set. 
A boundary component consists of all feasible points for which a given subset of the inequality constraints are active. 
The boundary components of the feasible set $\Phi(\Delta_\AA^\OO)$ are in one-to-one correspondence with the faces of $\Delta_{\mathcal A}^{\mathcal O}$ according to
\begin{equation}\label{eq:boundaryComponent}
    \left\{\pi\in \Delta_{\mathcal A}^{\mathcal O} :\ \pi(a|o)=0 \ \forall a\in A_o, o\in\OO
\right\} \longleftrightarrow 
\left\{\eta \in\Phi(\Delta_\AA^\OO) : \ \eta_{sa} = 0 \text{ for } a\in A_{g_\beta(s)}\right\} , 
\end{equation} 
where $A_o$ is a proper subset of $\AA$ for every $o\in\OO$, and $g_\beta(s)$ is the observation associated with state $s$. 
In particular, there is a boundary component associated to each tuple $(A_o)_{o\in\mathcal O}$ with $A_o\subsetneq \mathcal{A}$, $o\in\mathcal{O}$. 

We point out the following result, which allows us to ignore high-dimensional boundary components when searching for a maximizer of the reward. 
Recall that for an observation $o\in\mathcal O$, the cardinalities of the fibers of $g_\beta$ are denoted by $d_o=\lvert S_o\rvert$. 

\begin{theorem}[Existence of maximizers in low dimensional faces, \cite{montufar2017geometry}]
\label{theo: location_of_maximizers}
There exist $A_o\subsetneq \AA$ with $\lvert A_o^c\rvert\le d_o$, $o\in\OO$, such that the set $B$ described in \eqref{eq:B} contains a (globally optimal) solution of the problem~\eqref{eq:polyFormRewMax}. \end{theorem}

\begin{remark}\label{rem:relevantBC}
One approach to solving~\eqref{eq:polyFormRewMax} is to solve the critical equations over every boundary component and then selecting the critical point with the highest objective value. 
According to Theorem~\ref{theo: location_of_maximizers} there is a lower-dimensional boundary component that contains a global maximizer. 
This implies that, instead of considering the critical points in all $(2^{ n_\AA }-1)^{ n_\OO}$ boundary components, it is enough to consider those in the boundary components with $A_o\subsetneq \AA$ satisfying $\lvert A_o\rvert\ge n_\AA - d_o$. 
This reduces the number of boundary components 
that need to be checked to 
 \[ 
 \prod_{o\in\OO}\left( \sum_{l_o = \max(n_\AA-d_o, 0)}^{n_\AA-1} \binom{n_\AA}{l_o} \right), 
 \] 
which we call \struct{relevant} boundary components. 
Note that this number only depends on the number of actions $n_\AA$ and $d_o$ (the cardinality of the fibers of $g_\beta$).
\end{remark}

\subsection{Bounds via algebraic degrees of polynomial optimization}
With the description of the boundary components of the feasible set at hand, we can deduce upper bounds on the number of critical points over each of them based on the degrees of the defining equations and the degree of the objective function. 
\begin{theorem}[Bound on the algebraic degree] 
\label{theo: algebraic_degree_of_subproblem} 
Consider a POMDP with deterministic observations. 
Fix $A_o\subsetneq\AA$ for every $o\in\OO$ and set $n\coloneqq n_\SS n_\AA - n_\SS - \sum_o d_o\lvert A_o\rvert$ and $m\coloneqq \sum_o(d_o - 1)(\lvert A_o^c\rvert - 1)$, where we assume $n$ is not zero. 
Then the number of critical points of the linear function $\eta\mapsto \langle r, \eta\rangle$ over 
\begin{equation}\label{eq:B} 
\struc{B} \coloneqq \{\eta \in\mathcal L\cap\mathcal X : \ \eta_{sa} = 0 \text{ for  } a\in A_{g_\beta(s)}\} 
\end{equation} 
is upper bounded by $2^m \binom{n-1}{m-1}$. 
\end{theorem}
\begin{proof}
Recall from Corollary \ref{cor:Y} that $\Phi(\Delta_\AA^\OO)$ is defined in $\mathbb{R}^{\SS\times\AA}$ as an intersection of $n_\SS $ linear equations, $\sum_o(d_o-1)(n_\AA-1)$ quadratic equations of the form \eqref{eq:definingPolynomialsDeterministic}, and the linear inequalities $\eta\ge0$. 
It is not difficult to see that for any choice of $A_o\subsetneq \AA$, $o\in\OO$, 
the linear equations $\ell_s(\eta)=0$, $s\in\SS$ and $\eta_{sa}=0$, $a\in A_o$, $s\in S_o$, $o\in\OO$ are linearly independent. 
On the set $B$ given in~\eqref{eq:B} there are $\sum_o d_o\lvert A_o\rvert$ active linear inequalities with 
$A_o\subsetneq \AA$ for each $o\in\OO$
, and hence $B$ is contained in an affine space of dimension $n= n_\SS n_\AA - n_\SS -\sum_o d_o\lvert A_o\rvert$. 
Further, given these linear equations, the quadratic equations 
 \[ 
 p^o_{sa}(\eta) = \eta_{sa} \sum_{a'\in \mathcal A} \eta_{s_oa'} - \eta_{s_oa} \sum_{a'\in \mathcal A} \eta_{sa'} = 0 
 \]
are redundant for all $a\in A_o$, $s\in S_o$. 
By choosing $a_o\in A_o^c$ in Corollary~\ref{cor:Y} for every $o\in\OO$ there remain $ n_\AA - \lvert A_o\rvert - 1$ non-redundant quadratic equalities for every $s\in S_o\setminus\{s_o\}$. 
Therefore, we get $m = \sum_o (d_o-1)(\lvert A_o^c\rvert - 1)$ non-redundant quadratic equalities. 
By Theorem~2.2 and Corollary~2.5 in~\cite{Nie2009AlgebraicDO} the algebraic degree for the optimization of the linear function $r\in\mathbb R^{\SS\times\AA}$ over an $n$-dimensional affine space subject to $m$ non-redundant quadratic constraints is upper bounded by $2^m\binom{n-1}{m-1}$. 
\end{proof}

With Theorem~\ref{theo: algebraic_degree_of_subproblem} we can provide upper bounds for the number of critical points of the optimization problem~\eqref{eq:polyFormRewMax}. Indeed, the number of critical points over the interior 
\begin{equation}\label{eq:intB}
    \{\eta \in\mathcal L\cap\mathcal X : \ \eta_{sa} = 0 \text{ for all } a\in A_{g_\beta(s)}, \eta_{sa} > 0 \text{ otherwise }\}
\end{equation}
of a boundary component is clearly upper bounded by the number of critical points over $B$ defined in~\eqref{eq:B}. This bound over the individual boundary components can be summed to obtain an upper bound on the number of critical points of the polynomial optimization problem~\eqref{eq:polyFormRewMax} (see also~\cite{Nie2009AlgebraicDO}). Note that the Zariski closure of the interior of a boundary component defined in~\eqref{eq:intB} is contained in $B$ but might be a strict subset.
Similarly, a bound on the number of critical points over the relevant boundary components can be established.

\begin{remark}[Tighter bounds via polar degrees]
Since the problem~\eqref{eq:polyFormRewMax} has a linear objective, {under weak assumptions} the number of critical points over every boundary component is upper bounded by the polar degree of the associated variety. This approach may yield tighter bounds as demonstrated in the special case of a blind controller with two 
actions, i.e., a system with one observation and two actions in~\cite{mueller2021geometry}.
The authors obtain an upper bound linear in $n_\SS$ compared to the exponential upper bound of $ n_\SS \cdot 2^{ n_\SS -1} + 2$ that follows from Theorem~\ref{theo: algebraic_degree_of_subproblem}. 
A refinement of Theorem~\ref{theo: location_of_maximizers} for the case of mean rewards was also presented in \cite[Theorem~2]{montufar2019task}, which would be worthwhile studying from an algebraic standpoint.
\end{remark}

\subsection{Evaluation of the bounds}  

In Table~\ref{table:bounds} we present the upper bounds on the number of critical points for problems of different size. 
We compare the bound on the total number of critical points obtained by iterating Theorem~\ref{theo: algebraic_degree_of_subproblem} over all boundary components and the one iterating only over the relevant components described in Theorem~\ref{theo: location_of_maximizers}. In addition, we report the total and relevant number of boundary components discussed in Remark~\ref{rem:relevantBC}. 
Both, the number of boundary components and the upper bound on the number of critical points, depend on $n_\SS, n_\AA$, and the tuple $(d_o)_{o\in\OO}$. 
The two extreme cases for the tuple $(d_o)_{o\in\OO}$, namely $(n_\SS)$ and $(1, \ldots, 1)$, correspond to a \struct{blind controller}, i.e., all states map to the same observation, and the \struct{fully observable case}, i.e., states and observations are in one-to-one correspondence, respectively. 
The bounds are independent of the specific $\alpha$, so long as Assumption~\ref{ass:positivity} is satisfied.

\begin{center}
\begin{table}[ht]
\centering
\footnotesize 
\renewcommand{\arraystretch}{1.1}
\begin{tabular}{|M{0.8cm}|M{0.8cm}|M{2.6cm}|M{2cm}M{2cm}|M{2cm}M{2cm}|}
 \hline 
\multirow{2}{*}{$n_\SS$} & \multirow{2}{*}{$n_\AA$} & \multirow{2}{*}{\parbox{2.6cm}{\centering partitions of $n_\SS$: $(d_o)_{o\in\OO}$}} & \multicolumn{2}{M{4cm}|}{Number of boundary components} & \multicolumn{2}{M{4cm}|}{Bound on number of critical points} \\ \cline{4-7} 
& & & \multicolumn{1}{c}{total} & relevant & \multicolumn{1}{c}{total} & relevant \\ \hline
\multicolumn{1}{|c|}{\multirow{3}{*}{$3$}} & \multicolumn{1}{c|}{\multirow{3}{*}{$2$}} & $(3)$ & 3 & 3 & 10 & 10 \\ 
 \multicolumn{1}{|c|}{} & \multicolumn{1}{c|}{} & $(2,1)$ & 9 & 6 & 10 & 8 \\ 
 \multicolumn{1}{|c|}{} & \multicolumn{1}{c|}{} & $(1,1,1)$ & 27 & 8 & 8 & 8 \\ 
 \hline
 \multicolumn{1}{|c|}{\multirow{5}{*}{$4$}} & \multicolumn{1}{c|}{\multirow{5}{*}{$3$}} & $(4)$ & 7 & 7 & 1419 & 1419 \\ 
 \multicolumn{1}{|c|}{} & \multicolumn{1}{c|}{} & $(3,1)$ & 49 & 21 & 2237 & 561 \\ 
 \multicolumn{1}{|c|}{} & \multicolumn{1}{c|}{} & $(2,2)$ & 49 & 36& 1265 & 153 \\ 
 \multicolumn{1}{|c|}{} & \multicolumn{1}{c|}{} & $(2,1,1)$ & 343 & 54& 1189 & 81 \\ 
 \multicolumn{1}{|c|}{} & \multicolumn{1}{c|}{} & $(1,1,1,1)$ & 2401 & 81& 81 & 81 \\ 
 \hline
 \multicolumn{1}{|c|}{\multirow{7}{*}{$5$}} & \multicolumn{1}{c|}{\multirow{7}{*}{$3$}} & $(5)$ & 7 & 7 & 9411 & 9411 \\ 
 \multicolumn{1}{|c|}{} & \multicolumn{1}{c|}{} & $(4,1)$ & 49 &21 & 23745 & 4257 \\ 
 \multicolumn{1}{|c|}{} & \multicolumn{1}{c|}{} & $(3,2)$ & 49 & 42& 13431 & 4371 \\ 
 \multicolumn{1}{|c|}{} & \multicolumn{1}{c|}{} & $(3,1,1)$ & 343 & 63& 24363 & 1683 \\  \multicolumn{1}{|c|}{} & \multicolumn{1}{c|}{} & $(2,2,1)$ & 343 &108 & 12159 & 459 \\ 
 \multicolumn{1}{|c|}{} & \multicolumn{1}{c|}{} & $(2,1,1,1)$ & 2401 & 162& 9195 & 243 \\ 
 \multicolumn{1}{|c|}{} & \multicolumn{1}{c|}{} & $(1,1,1,1,1)$ & 16807 & 243& 243 & 243 \\
 \hline
\end{tabular}
\vspace{0.2cm}
\caption{\label{table:bounds}
Listed are the number of boundary components and the upper bound on the number of critical points from Theorem~\ref{theo: algebraic_degree_of_subproblem} both over all boundary components and over the subset of relevant boundary components from Theorem~\ref{theo: location_of_maximizers} for problems of different size. }
\end{table}
\end{center}

In these examples we observe that restricting to the relevant boundary components significantly reduces the upper bound. 
This is reflected in the last two columns in the table. 
The difference is most notable when the fibers of $g_{\beta}$ have a small cardinality, i.e., only few states lead to the same observation. 
In the fully observable case, the relevant boundary components correspond to the vertices of $\Delta_{\mathcal{A}}^{\mathcal{O}}$. 
This is consistent with the fact that in the fully observable case the feasible set $\Phi(\Delta_\AA^\OO)$ is a polytope~\cite{derman1970finite} and hence the optimization problem~\eqref{eq:polyFormRewMax} is a linear program, for which the solutions are attained at the vertices. 
On the other hand, in the case of a blind controller (with a single observation $o$), all boundary components are relevant since $d_o=n_\SS$.

\section{Numerical methods for the optimization of decision rules}\label{sec: optimizing decision rules} 

In POMDPs reward optimization over the set of memoryless stochastic policies~\eqref{eq:rewardMaximization} is known to be hard in theory (NP-hard \cite{vlassis2012computational}) and also difficult in practice as the reward function $R$ is nonconvex and has sub-obtimal strict local optima~\cite{poupart2011analyzing,bhandari2019global}. 
In this section we discuss how the geometric description of reward optimization facilitates computational approaches based on numerical algebra. 
We derive polynomial systems for the critical points, globally from the Karush-Kuhn-Tucker (KKT) conditions, and separately for each boundary component of $\Phi(\Delta_\AA^\OO)$ from the Lagrangian criterion. 
For different choices of $n_\SS$, $n_\AA$, and generic data (i.e., generic $\alpha, \mu$ and $r$), we compute the complex and real solutions of the KKT and Lagrangian systems, and compare the number of solutions with the theoretical upper bounds established in Section~\ref{sec: complexity of the problem}. 
Finally, we compare these approaches with other popular methods from constrained optimization: the interior point solver \texttt{Ipopt} and convex relaxations via the moment-SOS-approach. 

\subsection{Critical equations and computation} 
\paragraph{The KKT critical point equations} 
A standard approach for constrained optimization problems are the KKT conditions~\cite{KKT}, which provide necessary conditions of stationary points under certain regularity conditions; see e.g.~\cite{Abadie:KKT,bertsekas1997nonlinear,Bazaraa:Sherali:Shetty:KKT}. 
If both the constraints and objective are polynomial, the KKT conditions form a polynomial system, which can be solved using various numerical algebraic methods. 
 
Applied to our problem, the KKT conditions reduce to the following polynomial system in $\eta\in\mathbb R^{\SS\times\AA}_{\ge0}$
 with multipliers $\struc{\lambda}\in\mathbb R^\SS, \struc{\nu^o_{sa}}\in\mathbb R, \struc{\kappa}\in\mathbb R_{\geq 0}^{\SS\times\AA}$: 
\begin{align}\label{eq:KKT conditions}
 \begin{split}
 \text{Primal feasibility: } \quad & \ell_s(\eta) = 0 \text{ for } s\in\SS, \\ 
 & p_{sa}^o(\eta) = 0 \text{ for } o\in \OO, a\in\AA\setminus\{a_o\}, s\in S_o\setminus\{s_o\}, \\
 \text{Complementary slackness: }\quad & \kappa_{s_oa}\eta_{s_oa} = 0 \text{ for all } s_o, a, \\
 \text{Stationarity: }\quad & r +
 \sum_{s} \lambda_s\nabla \ell_s(\eta) + \sum_{o, s, a} \nu_{sa}^o \nabla p^o_{sa}(\eta)
 + \kappa = 0, 
 \end{split}
\end{align} 
where $a_o\in\AA$ and $s_o\in S_o$ for every $o\in\mathcal O$ are fixed arbitrarily. 
Here we have included the primal feasibility $\eta_{sa} \geq 0$ for $s\in\SS, a\in\AA$ and the dual feasibility $\kappa_{sa} \geq0$ for $s\in\SS, a\in\AA$ in the definition of the search space for $\eta$ and $\kappa$. 

The number of linear constraints $\ell_s$ 
is $n_\SS$, 
while the number of polynomial constraints 
$p^o_{sa}$ is $(n_\AA-1)\sum_{o\in\OO} (d_o-1) = (n_\AA-1)(n_\SS-n_\OO)$. 
Due to the symmetry of the effective policies, 
there are only $n_\OO n_\AA$ inequalities $\eta_{s_oa} \geq 0$ for each $a\in\AA, o\in\OO$. 
Hence the dimension of the square KKT system~\eqref{eq:KKT conditions} is 
 \[ n_\SS n_\AA + n_\SS + (n_\AA - 1)(n_\SS - n_\OO) +n_\OO n_\AA = 2n_\SS n_\AA + n_\OO. \]

In this setting, we can verify that the linear independence constraint qualification is satisfied. 
Given an element $\eta^*$ in the feasible set $\Phi(\Delta_\AA^\OO)$, 
it suffices to verify the linear independence of the gradients of the active inequality constraint functions and the equality constraints at $\eta^*$. 
Notice that under the pullback along the birational morphism $\Psi^{-1}$ the constraints are identified with affine-linear functions. 
Checking the linear independence of their gradients can be done by counting the dimension of the faces. \\ 

\paragraph{The Lagrange critical point equations over boundary components} 
Alternatively to solving the KKT system, one can compute the critical equations given by the Lagrange criterion over every boundary component individually. 
If there are no inequality constraints, the KKT equations specialize to the Lagrange multiplier equations. 
Consider a boundary component $B$ in~\eqref{eq:B} for a choice of $A_o\subsetneq\AA$ for every $o\in\OO$, and consider the optimization problem over $B$. 
This amounts setting $\eta(s,a) = 0$ for $a\in A_o$ whenever $g_\beta(s) = o, o\in\mathcal O$, which reduces optimization to a subspace of $\mathbb R^{\SS\times\AA}$. We denote the new primal variables by $\struc{\hat\eta}$. 
Similarly, we denote the restriction of $\ell_s$ and $p^{o}_{sa}$ to this space by $\struc{\hat{\ell}_s}$ and $\struc{\hat{p}^{o}_{sa}}$ and the projection of $r$ onto this space (i.e., the vector obtained by dropping the indices which are set to zero in $\eta$) by $\struc{\hat r}$. In the lower dimensional variables $\hat \eta$ for a given $B$ the Lagrange system becomes 
\begin{align}\label{eq:Lagrange equations}
 \begin{split}
 \text{Feasability: } \quad & \hat{\ell}_s(\hat\eta) = 0 \text{ for } s\in\SS, \\ 
 & \hat{p}_{sa}^o(\hat\eta) = 0 \text{ for } o\in \OO, a\in\AA\setminus\{a_o\}, s\in S_o\setminus\{s_o\}, 
 \\
  \text{Stationarity: }\quad & \hat r +
  \sum_{s} \lambda_s\nabla \hat{\ell}_s(\hat\eta) + \sum_{o, s, a} \nu_{sa}^o \nabla \hat{p}^o_{sa}(\hat\eta)
 = 0,
 \end{split}
 \end{align}
 where $a_o\in A_o^c$ and $s_o\in S_o$ are fixed arbitrarily for every $o\in\mathcal O$. 
The dimension of the primal variable $\hat\eta$ is $n_\SS n_\AA - \sum_{o} d_o \lvert A_o\rvert$, the dimension of the Lagrange multipliers $\lambda$ is $n_\SS$ and of $\nu
$ is $\sum_o(d_o-1)(\lvert A_o^c\rvert-1)$ (see also proof of Theorem~\ref{theo: algebraic_degree_of_subproblem}). Overall, the Lagrange system~\eqref{eq:Lagrange equations} is a square polynomial system of dimension
\[ 
2n_\SS n_\AA - (n_\AA-1) n_\OO - \sum_{o} (2d_o-1)\lvert A_o\rvert . 
\]

\begin{remark}[Lagrange vs KKT system] 
\label{remark:LagrangevsKKT}
It is easy to see that every real solution of the KKT system satisfying the primal and dual inequality constraints $\eta\ge0, \kappa\ge0$
is a solution of the Lagrange system over a boundary component, namely the boundary component defined by the zeros of $\eta$; see Figure~\ref{fig:KKTvsLagrange} for an illustrated example of this situation. 
When solving the KKT system~\eqref{eq:KKT conditions}, usually one solves the system of equations without the nonnegativity conditions $\eta\ge0$ and $\kappa\ge0$ and then selects the nonnegative solutions. 
Note that every solution of the Lagrange system over a boundary component appears as the solution of the KKT system without the nonnegativity constraints. 
Hence, solving the KKT system gives at least as many solutions as solving the Lagrange system over every boundary component. 
\begin{figure}
    \centering
    \begin{tikzpicture}[scale=0.8]
    \begin{scope}[overlay, scale=2]
     \clip [saveuse path={plot p1}{plot[smooth, samples=100, domain=-2.1:2.1] (\x, -.7 - .6*\x*\x + .3*\x*\x*\x*\x)}]
     ;
     \clip [saveuse path={plot p1.1}{plot[smooth, samples=100, domain=-2.1:2.1] (\x, .7 +  .1*\x*\x)}]
     -- (5,-5) -- (-5,-5) -- cycle 
     ;
     \clip [saveuse path={plot p3}{ plot[smooth, samples=100, domain=-2.1:2.1, variable=\y] (-\y*\y+1.5, \y)}]
     ;
     \clip [saveuse path={plot p3}{ plot[smooth, samples=100, domain=-2.1:2.1, variable=\y] (\y*\y-1.5, \y)}]
     ;

    \clip (0,0) circle [radius=3.] [draw, preaction={draw,fill=gray!50}]
    ;
  
\end{scope}
    
      \draw[->, line width=0.5] (0, 3) -- (0, 4) node[right]{} 
      ;
     
      \draw[scale=2, domain=-2.1:2.1, smooth, variable=\x, line width=0.5] plot ({\x}, 
      {0.3*(\x-1)*(\x-1)*(\x+1)*(\x+1) - 1});
      \draw[scale=2, domain=-4.4:4.4, smooth, variable=\x, line width=0.5]  plot ({\x}, {.7 +  .1*\x*\x});
      \draw[scale=2, domain=-1.5:2.45, smooth, variable=\y, line width=0.5]  plot ({\y*\y-1.5}, {\y});
      \draw[scale=2, domain=-1.5:2.45, smooth, variable=\y, line width=0.5]  plot ({-\y*\y+1.5}, {\y});
      \filldraw[scale=2, green] (0,-.7) circle (0.04);
      \filldraw[black, scale=2] (-1,-1) circle (0.04);
      \filldraw[black, scale=2] (1,-1) circle (0.04);
      \filldraw[blue, scale=2] (0,1.22) circle (.04);
      \filldraw[blue, scale=2] (0,-1.22) circle (.04);
      \filldraw[black, scale=2] (-2.02,1.87) circle (0.04);
      \filldraw[black, scale=2] (2.02,1.87) circle (0.04);
       \filldraw[green, scale=2] (-0.69,-0.91) circle (0.04);
       \filldraw[green, scale=2] (0.69,-0.91) circle (0.04);
      \filldraw[red, scale=2] (0.,0.7) circle
      (0.04);
      \filldraw[red, scale=2] (-0.9,0.78) circle
      (0.04);
      \filldraw[red, scale=2] (0.9,0.78) circle
      (0.04);
      \filldraw[blue, scale=2] (-1.9,1.06) circle
      (0.04);
      \filldraw[blue, scale=2] (1.9,1.06) circle
      (0.04);
      \filldraw[black, scale=2] (4.08,2.36) circle (0.04);
      \filldraw[black, scale=2] (-4.08,2.36) circle (0.04);
\end{tikzpicture}  

\caption{Schematic illustration of the feasible region (gray) and objective gradient (arrow) of a polynomially constrained linear program showing (i) the solutions of the KKT system checking only primal ($\eta\ge0$) inequality constrains (red and green); and checking primal ($\eta\ge0$) and dual ($\kappa\ge0$) inequality constrains  (red), (ii) the solutions of the KKT system without checking inequalities (red, green, black and blue), (iii) the positive ($\eta\ge0$) solutions of the Lagrange systems over all boundary components (red and green)
(iv) all solutions of the Lagrange systems over all boundary components (red, green and black). } 
\label{fig:KKTvsLagrange}
\end{figure}

\end{remark}

\paragraph{Computation} 
The optimization problem \eqref{eq:polyFormRewMax} can be solved using several methods: 
\begin{itemize}[leftmargin=*] 
\item First, we use the numerical algebra package \texttt{HomotopyContinuation.jl}~\cite{HomotopyContinuation} to solve the KKT system~\eqref{eq:KKT conditions} and the Lagrange system~\eqref{eq:Lagrange equations} of each boundary component. 
This automatically certifies the results \cite{breiding2021certifying}, meaning that for every returned solution it is guaranteed that there exists a unique true solution in a small neighborhood. 
From the returned solutions to the critical equations, we then just need to select the real ones that satisfy the primal inequality constraints $\eta_{s,a}\geq 0$, and among them the one that has the maximum objective value. 
\item 
Alternatively, we solve a convex relaxation of the polynomial optimization problem. Namely we relax the problem to a semidefinite program (SDP) via the moment-SOS-approach that is implemented in the freeware \texttt{GloptiPoly3}~\cite{gloptipoly}, and solve the SDPs using the numerical solver \texttt{Mosek}; see~\cite{DA2021} for details. 
We note that \texttt{GloptiPoly3} builds upon a hierarchy of moment/SOS programs (also called Lasserre hierarchy), which allows to approximate the optimal value arbitrarily close, and can be used to test optimality and extract global optimizers. We use this key feature to check if our methods reach global optimality. 
\item 
We may also solve the constrained optimization problem~\eqref{eq:polyFormRewMax} using the interior point solver \texttt{Ipopt}~\cite{wachter2006implementation}, which is a local optimization method for large-scale nonlinear optimization, an approach recently pursued in~\cite{mueller2022rosa}. 
\end{itemize} 
 
\subsection{Experiments}
\paragraph{Description of the experiments}
We test our computational approach on random POMDPs of different sizes. 
To this end we first specify the number of states $n_\SS$, the number of actions $n_\AA$, and the number of states aggregated in each observation $(d_o)_{o\in\OO}$ with $\sum_o d_o=n_\SS$. 
For each specification of these values, we generate $20$ random problems as follows. 
We sample the initial state distribution $\mu$ and the transition probabilities $\alpha(\cdot|s,a)$, $(s,a)\in\SS\times\AA$ from a uniform distribution on the simplex $\Delta_{\SS}$, 
and sample the instantaneous reward vector $r\in\mathbb R^{\SS\times\AA}$ from a standard Gaussian distribution. 
We use the same random data for each of the two approaches, KKT and Lagrange over boundary components.

\newcommand{\pa}[1]{(#1)}
\newcommand{\sd}[1]{\!\!\raisebox{.2ex}{\scriptsize\textcolor{gray}{$\pm#1$}}} 

\begin{table}[h]
\centering
\footnotesize

\def\arraystretch{1.1}
\setlength{\tabcolsep}{2pt}
\begin{tabular}{|c|c|c|c c c|c c c|c c c|} 
\hline
 \multirow{2}{*}{$n_\SS$} & \multirow{2}{*}{$n_\AA$} & \multirow{2}{*}{$(d_o)_{o\in\OO}$} & \multicolumn{3}{c|}{\begin{tabular}[c]{@{}c@{}}KKT  
 \end{tabular}} & \multicolumn{3}{c|}{Lagrange (all)} & \multicolumn{3}{c|}{\begin{tabular}[c]{@{}c@{}}Lagrange  (relevant)\end{tabular}} \\ 
 \cline{4-12} 
 & & & complex\!\! & real & positive & complex\!\! & real & positive & complex\!\! & real & positive 
 \\ \hline
\multirow{3}{*}{$3$} & \multirow{3}{*}{$2$} & \pa{3} & $6$ \sd{0} & 4.4 \sd{1.2} & 2.1 \sd{0.3} & 6 \sd{0} & 4.4 \sd{1.2} & 2.1 \sd{0.3} & 6 \sd{0} & 4.4 \sd{1.2} & 2.1 \sd{0.3} \\
 &  & \pa{2,1} & 12 \sd{0} & 10.1 \sd{1.9} & 4.25 \sd{0.44} & 10 \sd{0} & 8.2 \sd{1.9} & 4.25 \sd{0.44} & 8 \sd{0} & 6.7 \sd{1.6} & 4.25 \sd{0.44} \\
 &  & \pa{1,1,1} & 20 \sd{0} & 20 \sd{0} & 8 \sd{0} & 8 \sd{0} & 8 \sd{0} & 8 \sd{0} & 8 \sd{0} & 8 \sd{0} & 8 \sd{0} \\ \hline 
\multirow{5}{*}{$4$} & \multirow{5}{*}{$3$} & \pa{4} & 45 \sd{0} & 17.1 \sd{4.3} & 4.3 \sd{1.3} & 45 \sd{0} & 17.1 \sd{4.3} & 4.3 \sd{1.3} & 45 \sd{0} & 17.1 \sd{4.3} & 4.3 \sd{1.3} \\
& & \pa{3,1} & 150 \sd{0} & 79 \sd{11} & 11 \sd{1.9} & 129 \sd{0} & 68.7 \sd{9.7} & 11 \sd{1.9} & 81 \sd{0} & 41.6 \sd{8.5} & 10.9 \sd{1.8} \\
& & \pa{2,2} & 281.6 \sd{0.75} & 154 \sd{16} & 13.9 \sd{4.7} & 263 \sd{0} & 136 \sd{16} & 13.9 \sd{4.7} & 153 \sd{0} & 89 \sd{10} & 13.65 \sd{4.3} \\
& & \pa{2,1,1} & 381.2 \sd{0.7} & 292 \sd{23} & 31.5 \sd{4.3} & 216 \sd{0} & 168 \sd{16} & 31.5 \sd{4.3} & 81 \sd{0} & 68 \sd{11} & 30.9 \sd{4.0} \\
& & \pa{1,1,1,1} & 495 \sd{0} & 495 \sd{0} & 81 \sd{0} & 81 \sd{0} & 81 \sd{0} & 81 \sd{0} & 81 \sd{0} & 81 \sd{0} & 81 \sd{0} \\ \hline
\multirow{6}{*}{$5$} & \multirow{6}{*}{$3$} & \pa{5} & 71 \sd{0} & 21.4 \sd{6} & 3.7 \sd{0.98} & 71 \sd{0} & 21.4 \sd{6} & 3.7 \sd{0.98} & 71 \sd{0} & 21.4 \sd{6} & 3.7 \sd{0.98} \\
& & \pa{3,2} & 637.95 \sd{0.76} & 219 \sd{28} & 12.60 \sd{2.9} & 626 \sd{0} & 213 \sd{29} & 12.6 \sd{2.9} & 477 \sd{0} & 171 \sd{24} & 12.6 \sd{2.9} \\
& & \pa{4,1} & 269.85 \sd{0.49} & 99 \sd{20} & 11.9 \sd{3.3} & 234 \sd{0} & 87 \sd{18} & 11.9 \sd{3.3} & 144 \sd{0} & 52 \sd{13} & 11.55 \sd{2.6} \\
& & \pa{3,1,1} & 881.95 \sd{0.22} & 436 \sd{68} & 36 \sd{10} & 558 \sd{0} & 285 \sd{47} & 36 \sd{10} & 243 \sd{0} & 117 \sd{20} & 35.3 \sd{9.2} \\
& & \pa{2,2,1} & 1717.3 \sd{2.5} & 890 \sd{49} & 35.6 \sd{5.3} & 1260 \sd{0} & 624 \sd{56} & 36.5 \sd{7.1} & 459 \sd{0} & 244 \sd{25} & 35.7 \sd{6.6} \\
& & \pa{2,1,1,1} & 2269.9 \sd{3.9} & 1712 \sd{142} & 89 \sd{12} & 810 \sd{0} & 624 \sd{74} & 89.3 \sd{12.3} & 243 \sd{0} & 195 \sd{37} & 88.1 \sd{9.5} \\
& & \pa{1,1,1,1,1} & 3002.9 \sd{0.31} & 3002.9 \sd{0.3} & 243 \sd{0} & 243 \sd{0} & 243 \sd{0} & 243 \sd{0} & 243 \sd{0} & 243 \sd{0} & 243 \sd{0} \\ \hline
\end{tabular}
\caption{
\label{table:numberCriticalPointsKKT}
\label{table:critPointsLagrange}
Mean and standard deviation of the number of solutions of the KKT system~\eqref{eq:KKT conditions}, the Lagrange system~\eqref{eq:Lagrange equations} over all boundary components, and the Lagrange system over the relevant boundary components, 
for 20 random POMDPs with the indicated number of states $n_\SS$, actions $n_\AA$, and state aggregation partition $(d_o)_{o\in\OO}$. 
In our setting, positive solutions are feasible solutions. 
}
\end{table}

\paragraph{Discussion of the results} 

In this section we discuss the experimental results on the number of solutions obtained by solving the KKT and Lagrange systems introduced above. In Table~\ref{table:numberCriticalPointsKKT} we report the average and standard deviation of the number of complex, real, and positive solutions returned in each case. Note that in our setting, positive solutions (i.e., solutions satisfying $\eta\ge0$) are (primal) feasible solutions. We also compare the performance and the computational times of these methods with
convex relaxations and interior point methods.

\medskip
Following the discussion in Remark \ref{remark:LagrangevsKKT}, we start by comparing the number of solutions of the KKT and the Lagrange systems. In Table~\ref{table:numberCriticalPointsKKT} we see that the KKT system has at least as many complex solutions as the Lagrange systems over all boundary components. 
This is consistent with our previous discussion, since, as we have pointed out, any solution of the Lagrange system over a boundary component is a solution of KKT. 
Moreover, we observe that KKT and Lagrange over all boundary components have in general the same number of positive solutions (see Remark \ref{remark:LagrangevsKKT} and Figure \ref{fig:KKTvsLagrange}). 

\medskip
It is also worth noting the difference between the number of complex, real, and positive solutions. That is, in Table~\ref{table:numberCriticalPointsKKT} we observe in general that for the three types of systems there is a drop between the number of complex solutions and the number of real and positive solutions. 
However, we find an exception to this in the Lagrange system for fully observable systems {($d_o=(1,\ldots,1)$)}, where the number of complex, real, and positive solutions coincide. The reason for this is that in this case all boundary components are affine spaces, so only the zero-dimensional boundary components have a solution, and these correspond precisely to the $n_\mathcal{A}^{n_\mathcal{S}}$ vertices of the feasible set. 

We also observe that the number of complex solutions has a much smaller variance than the number of real real or positive ones. 
This is expected, since choosing the coefficients of polynomial systems randomly gives the same number of complex solutions with probability one. In fact, the number of complex solutions for the Lagrange system has no variance across the different random parameters. Still, we see a small variance in the number of complex KKT solutions, which we attribute to numerical instability which can prevent the software package \texttt{HomotopyContinuation.jl} from finding all solutions to the KKT system. 
In contrast to the complex case, the variance on the number of real and positive solutions is not due to numerical errors. This is a typical phenomenon in polynomial systems, and one of the possible limitations of classic algebraic methods when one wants to estimate the number of real solutions of a system.

\medskip
In the following we compare the experimental results presented in Table~\ref{table:numberCriticalPointsKKT} with the theoretical upper bounds shown in Table~\ref{table:bounds} and highlight two particular facts. 
{First notice that} in most cases the theoretical bound is significantly larger than the number of solutions of the Lagrange system. 
{Moreover, this} gap becomes particularly pronounced for problems where the fibers of $g_\beta$ are large. 
This clearly indicates that there is a {discrepancy} between the theoretical bounds and the algebraic degree of the optimization problem. Indeed, our bounds are based on theory for generic polynomials and hence we do not expect that they provide a tight estimate of the algebraic degree for the particular polynomials we are dealing with. 
{Here we also observe a particular behavior in} the case of fully observable systems where {the number of critical points of the} Lagrange {systems} agree with our bounds. 
On the other hand, we see that in some cases the number of solutions of KKT is larger than the bound, which agrees with our discussion on solutions of KKT and Lagrange systems in Remark~\ref{remark:LagrangevsKKT}. 

\medskip
In addition to analyzing the number of solutions of the KKT and Lagrange systems, we are interested in comparing the different solution methods for the optimization problem. Therefore, we compare the optimal solution found by solving these systems with \texttt{HomotopyContinuation.jl} with the one found by \texttt{Ipopt} and \texttt{GloptiPoly3}. 
Although \texttt{HomotopyContinuation.jl} is not guaranteed to find all solutions to the KKT and Lagrange systems, we observe that this approach yields a reward that is at least as high as the one obtained by the interior point method \texttt{Ipopt} and in a few instances strictly higher. In fact, solving the optimization problem with \texttt{GloptiPoly3} returns a certificate for the optimality of the result, which in all computed instances coincides with the optimal value obtained by solving the KKT and Lagrange systems with \texttt{HomotopyContinuation.jl}. 
That is, \texttt{GloptiPoly3} offers numerical evidence that they always provide globally optimal solutions. 
It is noteworthy that in all computed instances using \texttt{GloptiPoly3}, the optimal value of the optimization problem is already attained at the first order relaxation of the Lasserre hierarchy~\cite{Lasserre:FirstLasserreRelaxation}. We conjecture that objective value exactness for the first order relaxation of~\eqref{eq:polyFormRewMax} holds with high probability for generic input data. 
Since the size of the SDP depends very sensitively on the order of the relaxation, this conjecture would remedy one of the major drawbacks of the SDP relaxation method.

\medskip
Finally, we observed that solving the Lagrange equations only over the relevant boundary components is up to two orders of magnitude faster than solving them over all boundary components. The improvements become more pronounced when $g_\beta$ has small fibers in which we can exclude more faces by the means of Theorem~\ref{theo: location_of_maximizers}; see also Table~\ref{table:bounds}. 
The computation times for the solution of the KKT system are on the same order to magnitude as the computation time of the solution of the Lagrange systems over all boundary components. KKT was slightly faster when $g_\beta$ has small fibers and slightly slower when $g_\beta$ has large fibers.

\FloatBarrier 

\paragraph{Reproducibility statement} 
The computer code for our experiments is publicly available at  {\small\url{https://github.com/marinagarrote/Algebraic-Optimization-of-Sequential-Decision-Rules}}.
We conducted our experiments using \texttt{Julia}~\cite{bezanson2017julia} version 1.7.0, an open source programming language under the MIT license.    
We used the Julia package \texttt{HomotopyContinuation.jl} version 2.6.3, which is freely available for personal use under the MIT license, 
and \texttt{Ipopt.jl} version 0.7.0, {which is a Julia interface to the \texttt{COIN-OR} nonlinear solver \texttt{Ipopt} available under the EPL (Eclipse Public License) open-source license}. 
The convex relaxation is computed via the freeware \texttt{GloptiPoly3} implemented in Matlab, for which there also exists an Octave implementation.

\section{Conclusion and Outlook}

Reward optimization in infinite-horizon discounted MDPs with state aggregation and stationary stochastic policies is equivalent to a polynomially constrained linear objective problem. We obtained a characterization of the feasible region of state-action frequencies for this problem as the intersection of a product of affine varieties of rank one matrices, an affine space and a simplex. 
Using this description as a polynomial optimization problem, we derived an upper bound on the number of critical points of the reward optimization problem. 

To solve the polynomial optimization problem, we considered KKT equations and the Lagrange system over individual boundary components, 
where we leveraged knowledge about the location of maximizers on lower dimensional boundary components. 
The relatively small number of solutions observed in the experiments indicate that there is room for refining the theory either to obtain tighter estimates of the algebraic degree or also tighter descriptions of the possible number of feasible solutions. 
Using a convex relaxation to an SDP we obtained empirical evidence that our approach of solving the critical equations provides global maximizers of the reward. This is in strong contrast to naive gradient optimization, which yields only locally optimal solutions for this problem. 
In our experiments, the first order relaxation produced exact objective values, which indicates an interesting direction for further research.

\section*{Acknowledgment} 
The authors are grateful to Bernd Sturmfels for insightful discussions. 
GM and JM have been supported by ERC Starting Grant 757983 and DFG SPP 2298 Grant 464109215. 
GM has been supported by NSF CAREER Award DMS-2145630. 
JM acknowledges support from the International Max Planck Research School for Mathematics in the Sciences (IMPRS MiS) and the Evangelisches Studienwerk Villigst e.V. 
%
\bibliographystyle{amsalpha} 
\bibliography{bib}

\end{document}